\documentclass[twoside]{amsart}
\usepackage{mathrsfs}
\usepackage{amsfonts}
\usepackage{amssymb}
\usepackage{amssymb}
\usepackage{amssymb}
\usepackage{amssymb}
\usepackage{amssymb}
\usepackage{amssymb}
\usepackage{amssymb}
\usepackage{amssymb}
\usepackage{amsmath}
\usepackage[all]{xy}

\usepackage{lastpage}

\RequirePackage{amsmath} \RequirePackage{amssymb}
\usepackage{amscd,latexsym,amsthm,amsfonts,amssymb,amsmath,amsxtra}
\usepackage[colorlinks=true, urlcolor=blue,bookmarks=true,bookmarksopen=true,
citecolor=blue]{hyperref}

    \newcommand{\BA}{{\mathbb {A}}} 
    \newcommand{\BC}{{\mathbb {C}}} 
     \newcommand{\BF}{{\mathbb {F}}}

    \newcommand{\CC}{{\mathcal {C}}}

    \newcommand{\CM}{{\mathcal {M}}}

    \newcommand{\CW}{{\mathcal {W}}}

     \newcommand{\fo}{{\mathfrak{o}}}  \newcommand{\fp}{{\mathfrak{p}}}

    \newcommand{\RU}{{\mathrm {U}}}

     \newcommand{\bx}{{\bf {x}}} 
     
      \newcommand{\bW}{{\bf {W}}}

     \newcommand{\Nm}{{\mathrm {Nm}}}
    \newcommand{\lenth}{{\mathrm {\lenth}}}

    \newcommand{\Gal}{{\mathrm{Gal}}} \newcommand{\GL}{{\mathrm{GL}}}
    \newcommand{\Hom}{{\mathrm{Hom}}} 
    \newcommand{\Ind}{{\mathrm{Ind}}}

\newcommand{\bt}{{\mathbf{t}}}

    \newcommand{\cond}{\mathrm{cond}} 
    \renewcommand{\Re}{{\mathrm{Re}}} 
    \newcommand{\Res}{{\mathrm{Res}}}

    \newcommand{\Sp}{{\mathrm{Sp}}}
    \newcommand{\diag}{{\mathrm{diag}}} \newcommand{\Mat}{{\mathrm{Mat}}}

 \newcommand{\SO}{{\mathrm{SO}}}
 
 \newcommand{\tr}{{\mathrm{tr}}}
  
\newcommand{\vol}{{\mathrm{vol}}}

    \newcommand{\wt}{\widetilde}

    \newcommand{\wpair}[1]{\left\{{#1}\right\}}

    \newcommand{\ov}{\overline}
    
    \newcommand{\incl}{\hookrightarrow}
    
     \newcommand{\ra}{\rightarrow}

    \theoremstyle{plain}

    \newtheorem{thm}{Theorem}[section] \newtheorem{cor}[thm]{Corollary}
    \newtheorem{lem}[thm]{Lemma}  \newtheorem{prop}[thm]{Proposition}

    \numberwithin{equation}{section}

 \newtheorem*{theorem*}{Local Converse Theorem for $\RU_{2r+1}$}
 \newtheorem*{thm*}{Stability of the local gamma factors for $\RU_{2r+1}$}
 \newtheorem*{thmn}{A ``new" local converse theorem for $\GL_{2r+1}$}

\usepackage[top=1in, bottom=1in, left=1.25in, right=1.25in]{geometry}

\title{A local converse theorem for $\RU_{2r+1}$}
\author{Qing Zhang}
\address{School of Mathematics, Sun Yat-Sen University, Guangzhou, China, 510275}
\email{qingzhang0@gmail.com}

\begin{document}

\maketitle


\begin{abstract}
Let $E/F$ be a quadratic extension of $p$-adic fields and $\RU_{2r+1}$ be the unitary group associated with $E/F$. We prove the following local converse theorem for $\RU_{2r+1}$: given two irreducible generic supercuspidal representations $\pi,\pi_0$ of $\RU_{2r+1}$ with the same central character, if $\gamma(s,\pi\times \tau,\psi)=\gamma(s,\pi_0\times \tau,\psi)$ for all irreducible generic representation $\tau$ of $\GL_n(E)$ and for all $n$ with $1\le n\le r$, then $\pi\cong \pi_0$. The proof depends on analysis of the local integrals which define local gamma factors and uses certain properties of partial Bessel functions developed by Cogdell-Shahidi-Tsai recently.
\end{abstract}

\section*{Introduction}
Let $E/F$ be a quadratic extension of $p$-adic local fields and $\RU_{2r+1}$ be a quasi-split unitary group of $2r+1$ variables, corresponding to $E/F$. In this paper, we prove the following 
\begin{theorem*} Let $\pi,\pi_0$ be two irreducible generic supercuspidal representations of $\RU_{2r+1}$ with the same central character. If $\gamma(s,\pi\times \tau,\psi)=\gamma(s,\pi_0\times \tau,\psi)$ for all irreducible generic representation $\tau$ of $\GL_n(E)$ and for all $n$ with $1\le n\le r$, then $\pi\cong \pi_0$.
\end{theorem*}
When $r=1$, the above theorem was proved in \cite{Ba97}.

In the above theorem, the local gamma factors are those defined from local functional equations of the local zeta integrals considered in \cite{BAS}. To the author's understanding, it is still not known whether these local gamma factors are the same as those defined using Langlands-Shahidi method.

In \cite{Jng}, D. Jiang proposed local converse conjectures for classical groups. Following the recent proof  of Jacquet's local converse conjecture for $\GL$  \cite{Ch, JLiu}, the local converse conjectures for many classical groups could be proved by considering the descent map from $\GL$ to the classical groups, for examples, see \cite{JngS03} for the $\SO_{2n+1}$ case and \cite{ST, M} for the $\RU(n,n)$ case. The $\Sp_{2n}$ case is announced in \cite{JngS12}, although the proof is not published. To the author's knowledge, the irreducibility of the descent from $\GL$ to $\RU_{2r+1}$ is not considered in the literature and thus our result is new. 

In \cite{Zh}, we proved local converse theorems for $\Sp_{2r}$ and $\RU(r,r)$ based on pure local analysis of the integrals which define the local gamma factors. In particular, the proof of \cite{Zh} is independent of the recently proved local converse theorem for $\GL$ \cite{Ch, JLiu}. The current paper is a sequel of \cite{Zh}. Here we use similar methods as in \cite{Zh} to give a proof of the above local converse theorem for $\RU_{2r+1}$.\\

Our method of proof can give a ``new" local converse theorem for $\GL_{2r+1}$ as discussed below.

Temporarily, let $E/F$ be a quadratic extension of global fields, and $\RU_{2r+1}(E/F)$ be the unitary group associated with $E/F$. Let $\BA_F$ (resp. $\BA_E$) the ring of adeles of $F$ (resp. $E$). Let $\pi$ be an irreducible generic cuspidal automorphic representation of $\RU_{2r+1}(\BA_F)$ and $\tau$ be an irreducible generic cuspidal automorphic representation of $G_n(\BA_F)$, where $G_n=\Res_{E/F}(\GL_n)$. In \cite{BAS}, Ben-Artzi and Soudry constructed a global zeta integral which represents the $L$-function of $\pi\times \tau$ at unramified places. If $v$ is an inert place of $F$, then $E_v/F_v$ is a quadratic extension of local fields and the group $\RU_{2r+1}$ at such a place is a local unitary group associated with $E_v/F_v$. The local functional equation gives a local gamma factor $\gamma(s,\pi_v\times \tau_v,\psi_v)$, which is the local gamma factor we used in the above local converse theorem for $\RU_{2r+1}$, see $\S$1 for more details. 

 If $v$ is a place of $F$ which splits over $E$, then $E_v\cong F_v\oplus F_v$, $\RU_{2r+1}(E_v/F_v)\cong \GL_{2r+1}(F_v)$, and $G_n(F_v)=\GL_n(F_v)\oplus \GL_n(F_v)$. At a split place $v$, $\pi_v$ is a generic irreducible representation of $\GL_{2r+1}(F_v)$ and $\tau_v$ is a pair of irreducible generic representations $(\tau_{1,v},\tau_{2,v})$ of $\GL_n(F_v)\oplus \GL_n(F_v)$. From the local functional equation of the local zeta integrals at a split place $v$, one can obtain a local gamma factor $\gamma(s,\pi_v\times \tau_v,\psi_v)$. Our method of proof can give the following version of the local converse theorem for $\GL_{2r+1}$.

\begin{thmn}
Let $F$ be a $p$-adic field and $\pi,\pi_0$ be two irreducible supercuspidal generic representations of $\GL_{2r+1}(F)$ with the same central character. If $\gamma(s,\pi\times \tau,\psi)=\gamma(s,\pi_0\times \tau,\psi)$ for all pairs $\tau:=(\tau_1,\tau_2)$ of irreducible generic representations of $\GL_n(F)$ and for all $1\le n\le r$, then $\pi\cong \pi_0$.
\end{thmn}

 The above ``new" local converse theorem is expected to be not new. In fact, for a pair of irreducible generic representations $\tau=(\tau_1,\tau_2)$ of $\GL_n(F)$, it is expected that $\gamma(s,\pi\times \tau,\psi)=\gamma(s,\pi\times \tau_{1},\psi)\gamma(s,\tilde\pi\times\tau_{2},\psi)$ up to a normalizing factor which only depends on $\tau_1,\tau_2$, where $\gamma(s,\pi\times \tau_i,\psi)$ is the standard local gamma factor for $\GL_{2r+1}\times \GL_n$ for $i=1,2$. If the expected property holds, the above ``new" local converse theorem for $\GL_{2r+1}$ in fact gives a new proof of Jacquet's local converse conjecture for $\GL_{2r+1}$, which was proved in \cite{Ch} and \cite{JLiu} recently. A consideration in the $\RU(r,r)$ case would give a new proof of Jacquet's conjecture for $\GL_{2r}$.

 Since the similarities of these two situations, in our paper, we will only give details of the proof in the non-split case, i.e., $\RU_{2r+1}$ case, and leave the proof in the split case, i.e., the above ``new" local converse theorem for $\GL_{2r+1}$ to the reader.\\

As a byproduct of our proof, we can prove the stability of the local gamma factors for $\RU_{2r+1}$:
\begin{thm*}Let $\pi,\pi_0$ be two irreducible generic supercupidal representations of $\RU_{2r+1}$ with the same central character. Then there exists an integer $l=l(\pi,\pi_0)$ such that for any quasi-character $\chi$ of $E^\times$, with $\cond(\chi)>l$, we have
$$\gamma(s,\pi\times \chi,\psi)=\gamma(s,\pi_0\times \chi,\psi).$$
\end{thm*}
Stability of local gamma factors are proved in various settings. For example, in \cite{CPSS08}, Cogdell-Shahidi-Piatetski-Shapiro proved the stability of local gamma factors arising from Langlands Shahidi method for any quasi-split groups. Since it is still not known the local gamma factors arising from the local functional equation are the same as those local gamma factors arising from Langlands-Shahidi method, our case is not covered by \cite{CPSS08}.\\

This paper is organized as follows. In $\S$1, we review the definition of local gamma factors for $\RU_{2r+1}\times \Res_{E/F}(\GL_n)$ for $n\le r$. In $\S$2, we review some background materials, including Howe vectors and Cogdell-Shahidi-Tsai's theory on partial Bessel functions, and then give an outline of our proof of the local converse theorem for $\RU_{2r+1}$. After some preparations in $\S3$, the local converse theorem for $\RU_{2r+1}$ is proved in $\S$4. We then prove the stability of the local gamma factors in $\S$5.

\section*{Acknowledgment} 
I would like to thank Jim Cogdell for useful discussions, constant encouragement and support over the years. I am very grateful to every member in the automorphic representation seminar in The Ohio State University, where I learned Cogdell-Shahidi-Tsai's results on partial Bessel functions.I appreciated the anonymous referee for his/her careful reading of the first draft and many helpful suggestions.

\section*{Notation}
Most of our notations follow from \cite{BAS}.

Let $F$ be a $p$-adic field and $E/F$ be a quadratic extension. Denote by $\fo_E$ (resp. $\fo_F$) the ring of integers of $E$ (resp. $F$), by $\fp_E$ (resp. $\fp_F$) the maximal ideal of $\fo_E$ (resp. $\fo_F$). Let $\varpi_E$ (resp. $\varpi_F$) a fixed uniformizer of $E$ (resp. $F$). Let $q_E=|\fo_E/\fp_E|$ and $q_F=|\fo_F/\fp_F|$.

  Denote by $z\mapsto \bar z, z\in E$ the nontrivial Galois action in $\Gal(E/F)$. Then we can define the trace and norm by $\tr_{E/F}(z)=z+\bar z$ and $\Nm_{E/F}(z)=z\bar z$ for $z\in E$. Denote $E^1$ the norm 1 elements of $E^\times$, i.e., $E^1=\wpair{z\in E^\times: z\bar z=1}$. For a positive integer $k$, denote $G_k=\Res_{E/F}(\GL_k)$, where $\GL_k$ is viewed as an algebraic group over $E$. Thus $G_k$ is an algebraic group over $F$ and $G_k(F)=\GL_k(E)$.

For a positive integer $k$, denote $$J_k=\begin{pmatrix}&1\\ J_{k-1}& \end{pmatrix}, J_1=(1).$$
Let $$\RU_k=\wpair{g\in G_k: {}^t\! \bar g J_k g=J_k}.$$

For $k=2n$ even, we denote by $Q_{2n}=M_{2n}\ltimes V_{2n}$ the Siegel type parabolic subgroup of $\RU_{k}$, where 
$$M_{2n}=\wpair{m(a):=\begin{pmatrix}a &\\ &a^* \end{pmatrix}\in \RU_{2n},a\in G_n},\quad V_{2n}=\wpair{u(x):=\begin{pmatrix} I_n &x\\ &I_n \end{pmatrix}\in \RU_{2n}}.$$
Here $a^*$ is determined by $a$ such that $m(a)\in \RU_{2n}$. One can check that $a^*=J_n {}^t\!\bar a^{-1} J_n
$. Let $N_{2n}$ be the upper triangular unipotent subgroup of $\RU_{2n}$. Denote $w_n=\begin{pmatrix} &I_n\\ I_n& \end{pmatrix}\in \RU_{2n}$.

In this paper, we will mainly concern $\RU_k$ when $k$ is odd, i.e., $k=2r+1$ for a positive integer $r$. Let $B_k=A_kN_k$

be the upper triangular Borel subgroup with maximal torus $A_k$ and maximal unipotent $N_k$. The center of $\RU_{2r+1}$ is consisting elements of the form $\diag(z,z,\dots,z,z),z\in E^1$. We will identify $E^1$ with the center of $\RU_{2r+1}$ via the map $z\mapsto \diag(z,z,\dots,z,z)$. A typical element $t\in A_{2r+1}$ has the form
$$t=z\diag(a_1,\dots,a_r,1,\bar a_r^{-1},\dots,\bar a_1^{-1}),z\in E^1,a_1,\dots,a_r\in E^\times.$$

Define characters $\alpha_i,1\le i\le r$ on $A_{2r+1}$ by 
$$\alpha_i(t)=a_i/a_{i+1},1\le i\le r-1; \alpha_r(t)=\alpha_r,$$
where $t=z\diag(a_1,\dots,a_r,1,\bar a_r^{-1},\dots,\bar a_1^{-1}).$
Denote $\Delta_r=\wpair{\alpha_i,1\le i\le r}$, which is the set of simple roots of $\RU_{2r+1}$. For a root $\beta$, denote by $U_\beta$ the root space of $\beta$ and we fix an isomorphism $\bx_\beta: \BF\ra U_\beta$, where $\BF=E$ or $F$, which depends on the root $\beta$.

Throughout the paper, we will focus on the group $\RU_{2r+1}$ for a fixed positive integer $r$ and consider its various subgroups. If $r$ is understood, we will drop the subscript $r$ from the notation. For example, the letter $A$ will be used to denote the maximal diagonal torus of $\RU_{2r+1}$ and $N$ will be used to denote the upper triangular maximal unipotent subgroup of $\RU_{2r+1}$ if $r$ is understood. 

For $n\le r$, we consider the embedding $\RU_{2n}\ra \RU_{2r+1}$
$$\begin{pmatrix}a&b \\ c&d \end{pmatrix}\mapsto \begin{pmatrix}I_{r-n} &&&&\\ &a&&b&\\ &&1&&\\ &c&&d&\\ &&&&I_{r-n} \end{pmatrix}.$$
From this embedding, we will identify $\RU_{2n}$ as a subgroup of $\RU_{2r+1}$ without further notice. Thus the element $w_n\in \RU_{2n}$ can be viewed as an element of $\RU_{2r+1}$ for $n\le r$. 

For $ X\in G_r$, following \cite{BAS}, we denote 
$$X^{\wedge}=\begin{pmatrix}X&&\\&1&\\&&X^* \end{pmatrix}.$$
For $n\le r$, denote $w_{n,r-n}=\begin{pmatrix} &I_n\\ I_{r-n} & \end{pmatrix}^\wedge.$ Note that $w_{n,r-n}^{-1}=w_{r-n,n}$. Denote 
$$\tilde w_n=w_{n,r-n}w_n w_{n,r-n}^{-1}=\begin{pmatrix}&&I_n\\ &I_{2(r-n)+1}&\\ I_n&& \end{pmatrix}.$$

For $n\le r$, denote by $P_n$ the parabolic subgroup of $\RU_{2r+1}$ with Levi subgroup
 $$L_n=\wpair{E^1\cdot\begin{pmatrix}a&&&\\ &a_{n+1}&&\\ &&\dots& \\ &&&a_r \end{pmatrix}^\wedge, a\in G_n, a_{n+1},\dots,a_r\in G_1}.$$
Here $E^1$ is still identified with the center of $\RU_{2r+1}$.

\section{The local gamma factors for $\RU_{2r+1}$}

We fix a non-trivial additive character $\psi_0$ of $F$ and put $\psi=\psi_0\circ \tr_{E/F}$. Note that $\psi(\bar z)=\psi(z)$ for $z\in E$. 
\subsection{Induced representation of $\RU_{2n}$}
Following \cite{BAS}, we denote by $Z_n$ the upper triangular unipotent subgroup of $G_n(F)=\GL_n(E).$ We consider the character $\psi_{Z_n}$ of $Z_n$ by 
$$\psi_{Z_n}(z)=\psi(z_{1,2}+\dots+z_{n-1,n}),z=(z_{ij})\in Z_n.$$
Let $(\tau,V_\tau)$ be an irreducible $\psi_{Z_n}^{-1}$-generic representation of $\GL_n(E)$. We fix a nonzero Whittaker functional $\lambda\in \Hom_{Z_n}(\tau,\psi_{Z_n}^{-1})$. Note that $M_{n}\cong \GL_n(E)$. Given a complex number $s\in \BC$, we consider the induced representation $I(s,\tau)=\Ind_{M_n\ltimes V_n}^{\RU_{2n}}(\tau_s \otimes 1_{V_n})$, where $\tau_s$ is the representation of $M_n$ defined by 
$$\tau_s(m(a))=|\det(a)|_E^{s-1/2}\tau(a).$$
Thus the space of $I(s,\tau)$ is consisting of all smooth functions $f_s:\RU_{2n}\ra V_\tau $ satisfying 
$$f_s(m(a)u(x)h)=|\det(a)|_E^{s+\frac{n-1}{2}}\tau(a)f_s(h),a\in \GL_n(E),u(x)\in V_n,h\in \RU_{2n}.$$
For $f_s\in I(s,\tau)$, we consider the $\BC$-valued function 
$$\xi_{f_s}(g,a)=\lambda(\tau(a)f_s(g)),g\in \RU_{2n},a\in \GL_n(E).$$
From the quasi-invariance property of $f_s$, we have 
$$\xi_{f_s}(m(z)u(x)h,I_n)=\psi_{Z_n}^{-1}(z)\xi_{f_s}(h,I_n),  z\in Z_n, u(x)\in V_n, h\in \RU_{2n}.$$
Denote $V(s,\tau,\psi^{-1})=\wpair{\xi_{f_s}| f_s\in I(s,\tau)}$. 

Denote by $\tau^*$ the representation of $\GL_n(E)$ defined by $\tau^*(a)=\tau(a^*)$. Recall that $a^*=J_r {}^t\! \bar a ^{-1} J_r$. There is a standard intertwining operator $M(s,\tau): V(s,\tau,\psi^{-1})\ra V(1-s,\tau^*, \psi^{-1})$ defined by 
$$ M(s,\tau)\xi_s(g,a)=\int_{V_n}\xi_s(w_n^{-1}ug,d_n u)du, g\in \RU_{2n}, a\in \GL_n(E),$$
where $d_n=\diag(-1,1,\dots,(-1)^n)\in \GL_n$. It is well-known that $M(s,\tau)$ is well-defined for $\Re(s)>>0$ and can be meromorphically continued to the whole $\BC$-plane. 

\subsection{The local gamma factor}
Recall that $N$ is the upper triangular unipotent subgroup of $\RU_{2r+1}$. Consider the generic character $\psi_N$ of $N$ defined by 
$$\psi_N(u)=\psi(u_{12}+\dots +u_{r,r+1}),u=(u_{ij})\in N.$$
Let $\pi$ be an irreducible $\psi_U$-generic representation of $\RU_{2r+1}$. Denote by $\CW(\pi,\psi_N)$ its Whittaker model. Let $n\le r$ be a positive integer and $\tau$ be an irreducible $\psi_{Z_n}^{-1}$-generic representation of $\GL_n(E)$. For $W\in \CW(\pi,\psi_N)$ and $\xi_s\in V(s,\tau,\psi^{-1})$, we consider the following local zeta integral
$$\Psi(W,\xi_s)=\int_{N_{2n}\setminus \RU_{2n}}\int_{\Mat_{(r-n)\times n}}W\left(w_{n,r-n}\begin{pmatrix}I_{r-n}&x\\ &I_n\end{pmatrix}^\wedge h w_{n,r-n}^{-1}\right) \xi_s(h,I_n)dxdh.$$
Recall that $N_{2n}$ denotes the upper triangular unipotent subgroup of $\RU_{2n}$. The above local zeta integrals were studied in \cite{BAS}, where it was proved that the above integral is absolutely convergent for $\Re(s)>>0$, is nonzero for certain choices of $W,\xi_s$, defines a rational function of $q_E^{-s}$, and computes $\frac{L(s,\pi\times \tau)}{L(2s,Asai, \tau)}$ if every data is unramified. \\
\noindent \textbf{Remark:}  Some special cases of the above integrals were considered earlier in \cite{GPS, Ba97} when $n=r=1$, and in \cite{Ta} when $n=r\ge 1$.

\begin{prop}\label{prop1.1}
There exists a meromorphic function $\gamma(s,\pi\times \tau,\psi)$ such that
$$\Psi(W,M(s,\tau)\xi_s)=\gamma(s,\pi\times \tau,\psi)\Psi(W,\xi_s),$$
for any $W\in \CW(\pi,\psi)$ and $\xi_s\in V(s,\tau,\psi^{-1})$.
\end{prop}
\begin{proof}
If $n=r=1$, this is proved in \cite[Corollary 4.8]{Ba97}. The general case in fact follows from the uniqueness of Bessel models for $\RU_{2r+1}$, which is proved in \cite{GGP,AGRS}. Since we could not find appropriate reference, we provide some details here. We fix $n,r$ with $n\le r$. Temporarily, we denote $P(X)=M(X)N(X)$ the parabolic subgroup of $\RU_{2r+1}$ with Levi subgroup
$$M(X)=\wpair{\begin{pmatrix}a&&\\ &g&\\ &&a^* \end{pmatrix},a\in \GL_{r-n}(E),g\in \RU_{2n+1} },$$
and unipotent subgroup $N(X)$. Here the notation $X$ is nothing but just an index symbol. Denote by $U_X$ the upper triangular unipotent subgroup of $\GL_{r-n}(E)$, which is embedded into $M(X)$ in the natural way. Denote by $\psi_X$ the character on $U_X$ defined by 
$$\psi_X(u_{ij})=\psi\left(\sum_{i=1}^{r-n-1}u_{i,i+1}\right), u=(u_{ij})_{1\le i,j\le r-n}\in U_X.$$
We also consider the character $\psi'_X$ of $N(X)$ defined by 
$$\psi_X'(n)=\psi(n_{r-n,r+1}), n=(n_{ij})_{1\le i,j\le 2r+1}\in N(X)\subset \RU_{2r+1}.$$
For example, if $n=1,r=2$, we have 
$$\psi'_X\left(\begin{pmatrix}1&x_1&x_2&x_3&x_4\\ &1&&&x_3'\\ &&1&&x_2'\\ &&&1&x_1'\\ &&&&1 \end{pmatrix} \right)=\psi(x_2).$$

Temporarily, we denote by $H$ the following subgroup of $\RU_{2r+1}$:
$$H=(U_X\cdot \RU_{2n})\ltimes N(X).$$
Recall that we always identify $\RU_{2n}$ with a subgroup of $\RU_{2r+1}$ and with this identification, $U_X\cdot \RU_{2n}$ is a subgroup of $M(X)$. In matrix form, we have 
$$H=\wpair{\begin{pmatrix}u&*&*&*&*\\ &a&&b&*\\ &&1&&*\\&c&&d&*\\ &&&&u^* \end{pmatrix},u\in U_X, \begin{pmatrix}a&b\\ c&d \end{pmatrix}\in \RU_{2n},a,b,c,d\in \Mat_{n\times n}(E)}.$$
One important property for the character $\psi_X'$ is that 
$$\psi_X'(ung)=\psi_X'(n), \forall n\in N(X), u\in U_X, g\in \RU_{2n},$$
where $U_X$ and $\RU_{2n}$ are viewed as subgroup of $\RU_{2r+1}$ in the usual way. From this fact, we can extend $\psi_X'$ to a character of $H$.
In particular, for an irreducible representation $\sigma$ of $\RU_{2n}$, there is a representation $\lambda_X\otimes \sigma$ of $H$ such that $\lambda_X|_{U_X}=\psi_X,\lambda_X|_{N(X)}=\psi'_X,$ and $\lambda_X|_{\RU_{2n}}=\sigma$, see \cite[$\S$12]{GGP}. The pair $(H,\lambda_X\otimes \sigma)$ is called a Bessel data of $\RU_{2n+1}$. Given an irreducible smooth representation $\pi$ of $\RU_{2r+1}$ and an irreducible smooth representation $\sigma$ of $\RU_{2n}$, one has 
$$\dim \Hom_H(\pi,\lambda_X\otimes \sigma)\le 1,$$
by the main theorem of \cite{AGRS} (which deals with the equal rank case, i.e., when $n=r$) and \cite[Theorem 15.1]{GGP} (which deduce the general case from the equal rank case). A nonzero element in $\Hom_H(\pi,\psi_X\otimes \sigma)$ defines an embedding $\pi\incl \Ind_H^{\RU_{2r+1}}(\lambda_X\otimes \sigma)$, which is called the Bessel model of $\pi$ associated with the data $\lambda_X\otimes \sigma$.

We now go back to the proof of our proposition. Denote by $\pi^j$ the representation of $\RU_{2r+1}$ defined by $\pi^j(g)=\pi(w_{n,r-j}g w_{n,r-n}^{-1})$. It is routine (but tedious) to check that both of the bilinear forms $(W,\xi_s)\mapsto \Psi(W,\xi_s)$ and $(W,\xi_s)\mapsto (W,M(s,\tau)\xi_s)$ defined elements in 
$$\Hom_H(\pi^j, \lambda_X\otimes I(s,\tau)).$$
Here we identified the representation $\pi^j$ with its Whittaker model.
Since the induced representation $I(s,\tau)$ is irreducible outside a finite number of $q_F^s$, we get $\dim \Hom_H(\pi^j, \lambda_X\otimes I(s,\tau))\le 1 $ outside a finite number of $q_F^s$ by the uniqueness of Bessel models. This gives us the local functional equation and the existence of local gamma factors.
\end{proof}
\noindent \textbf{Remark:} It is expected that the above local gamma factor is the same as the local gamma factor defined from Langlands-Shahidi method up to a normalizing factor. Kaplan addressed this question for some other classical groups in \cite{Ka}, but unfortunately the unitary group case is not included in \cite{Ka}.

\section{Howe vectors and Partial Bessel functions}
In this section, we collected some results on Howe vectors and the associated partial Bessel functions we will need for the proof of the local converse theorem. The exposition here is quite similar to \cite[$\S$2]{Zh}, and we refer the reader to \cite[$\S$2]{Zh} for more details.

In the following, we will fix a positive integer $r$ and consider the group $G=\RU_{2r+1}$. Recall that we have identified $E^1$ with the center of $G$. Let $\omega$ be a character of $E^1$ and let $C_c^\infty(G,\omega)$ be the space of compactly supported smooth functions $f$ on $G$ such that $f(zg)=\omega(z)f(g)$ for all $z\in E^1$. Note that if $\pi$ is a an irreducible supercuspidal representation of $G$ with central character $\omega$, then the space $\CM(\pi)$ of matrix coefficients of $\pi$ is a subspace of $C_c^\infty(G,\omega)$. 

Recall that $B=AN$ is the upper triangular Borel subgroup with torus $A$ and maximal unipotent subgroup $N$. Let $\psi$ be an unramified non-trivial character of $E$ and let $\psi_U$ be the corresponding generic character of $N$. Denote by $C^\infty(G,\psi_N,\omega)$ the space of functions $W$ on $G$ such that $W(zg)=\omega(z)W(g)$, $W(ug)=\psi_N(u)W(g)$ for all $z\in E^1, u\in N, g\in G$, and there exists an open compact subgroup $K$ (depending on $W$) of $G$ such that $W(gk)=W(g)$ for all $g\in G,k\in K$. Note that if $\pi$ is a $\psi_U$-generic smooth representation of $G$, then its Whittaker model $\CW(\pi,\psi_N)$ is a subspace of $C^\infty(G,\psi_N,\omega)$.

\subsection{Howe vectors}
Let $m$ be a positive integer and $K_m^G=(I_{2r+1}+\Mat_{(2r+1)\times (2r+1)}(\fp_E^m))\cap G.$
Let $$e_m=\diag(\varpi_E^{-2rm},\varpi_E^{-2(r-1)m},\dots,\varpi_E^{-2m},1,\bar\varpi_E^{2m},\dots,\bar \varpi_E^{2rm})\in A.$$
Let $H_m=e_mK_m^G e_m^{-1}$. One can check that 
$$H_m=\begin{pmatrix}1+\fp_E^m & \fp_E^{-m} &\fp_E^{-3m} & \dots & \fp_E^{-(4r-1)m}\\ 
                                      \fp_E^{3m}& 1+\fp_E^m & \fp_E^{-m} &  \dots & \fp_E^{-(4r-3)m} \\
                                      \fp_E^{5m} & \fp_E^{3m} & 1+\fp_E^m & \dots & \fp_E^{-(4r-5)m}\\
                                     
                                      \dots & \dots & \dots & \dots &\dots \\
                                      \fp_E^{(4r+1)m} & \fp_E^{(4r-1)m} & \fp_E^{(4r-3)m}& \dots & 1+\fp_E^m
                              \end{pmatrix}\cap \RU_{2r+1}(F).$$

Let $\psi$ be a nontrivial unramified character of $E$ as above. We can define a character $\tau_m$ of $K_m^G$ by 
$$\tau_m((k_{ij}))=\psi(\varpi_E^{-2m}(\sum_{i=1}^r)k_{i,i+1}), k=(k_{i,j})\in K_m^G.$$
One can check that $\tau_m$ is indeed a character. Define a character $\psi_m$ of $H_m$ by $\psi_m(h)=\tau_m(e_m^{-1}he_m), h\in H_m$. Let $N_m=N\cap H_m$. Note that $N=\cup_{m\ge 1} N_m$ and $\psi_m|_{N_m}=\psi_N|_{N_m}$.

Let $\omega$ be a character of $E^1$. Given a function $W\in C^\infty(G,\psi_N,\omega)$ with $W(I_{2r+1})=1$ and a positive integer $m>0$, we consider the function $W_m$ on $G$ defined by 
$$W_m(g)=\frac{1}{\vol(N_m)}\int_{N_m}\psi_m(u)^{-1}W(gu)du.$$
Let $C=C(W)$ be an integer such that $W$ is fixed by $K_C^G$ on the right side. Following \cite{Ba95,Ba97}, a function $W_m$ with $m\ge C$ is called a \textbf{Howe vector}.

\begin{lem}\label{lem2.1}
We have 
\begin{enumerate}
\item $W_m(I_{2r+1})=1;$
\item if $m\ge C$, then $W_m(gh)=\psi_m(h)W_m(g)$ for all $h\in H_m;$
\item if $m\ge k$, then 
$$W_{m}(g)=\frac{1}{\vol(N_m)}\int_{N_m}\psi_{N}(u)^{-1}W_k(gu)du.$$
\end{enumerate}
\end{lem}
The proof of the above lemma in similar situations can be found in \cite[Lemma 3.2]{Ba95} or \cite[Lemma 5.2]{Ba97}. The same proof applied in our case.

By Lemma \ref{lem2.2}(2), for $m\ge C$, the function $W_m$ satisfies the relation 
\begin{equation}\label{eq2.1}
W_m(ugh)=\psi_N(u)\psi_m(h)W_m(g), \forall u\in N, g\in G, h\in H_m.
\end{equation}
Due to this relation, we also call $W_m$ a partial Bessel function. 

\begin{lem}\label{lem2.2}
For $m\ge C, t\in A$, if $W_m(t)\ne 0$, then $\alpha_i(t)=1+\fp_E^m$ for all $i$ with $1\le i\le r$. 
\end{lem}
\begin{proof}
The proof is the same as in some other cases considered by Baruch, see \cite[Lemma 6.2.9]{Ba95} and \cite[Lemma 5.4]{Ba97}. We give some details here. Given $x\in \fp_E^{-m}$, we have the relation $$t\bx_{\alpha_i}(x)=\bx_{\alpha_i}(\alpha_i(t)x)t.$$
Since $\bx_{\alpha_i}(x)\in N_m$, and $\psi_m(\bx_{\alpha_i}(x))=\psi(x)$, by Eq.(\ref{eq2.1}) we get $$\psi(x)W_m(t)=\psi(\alpha_i(t)x)W_m(t).$$
Thus if $W_m(t)\ne 0$, we have $\psi((\alpha_i(t)-1)x)=1$ for all $x\in \fp_E^{-m}$. Since $\psi$ is unramified, we have $\alpha_i(t)\in 1+\fp_E^m$.
\end{proof}
\begin{cor}\label{cor2.3}
Let $t=z\diag(a_1,a_2,\dots,a_r,1,\bar a_r^{-1},\dots,\bar a_1^{-1})\in A$ with $z\in E^1, a_i\in E^\times, 1\le i\le r$, and $m\ge C$, then we have 
$$W_m(t)=\left\{\begin{array}{lll}\omega(z), & \textrm{if } a_i\in 1+\fp_E^m, 1\le i\le r,\\ 0, & \textrm{otherwise.} \end{array} \right.$$
\end{cor}
\begin{proof}
By Lemma \ref{lem2.2}, if $W_m(t)\ne 0$, we have $\alpha_i(t)\in 1+\fp_E^m$. Thus we get $a_1,\dots,a_r\in 1+\fp_E^m$ from the definition of $\alpha_i$. Then we have $\diag(a_1,\dots,a_r,1,\bar a_r^{-1},\dots,\bar a_1^{-1})\in H_m$. Then by Lemma \ref{lem2.1}, we have 
$$W_m(t)=W_m(z)=\omega(z)W_m(I_{2r+1})=\omega(z).$$
This concludes the proof.
\end{proof}
Given $f\in C_c^\infty(G,\omega)$, we consider 
$$W^f(g)=\int_N \psi_N^{-1}(u)f(ug)du.$$
Note that the above integral is well-defined since $Ng$ is closed in $G$ and $f$ has compact support in $G$. In fact we have $W^f\in C^\infty(G,\psi_U,\omega)$. We assume that there exists a function $f\in C_c^\infty(G,\omega)$ such that $W^f(I_{2r+1})=1$. For sufficiently large $m$, we consider the corresponding partial Bessel function 
\begin{equation}\label{eq2.2}B_m(g,f):=(W^f)_m(g)=\frac{1}{\vol(N_m)}\int_{N\times N_m}\psi_N^{-1}(u)\psi_m^{-1}(u')f(ugu')dudu'.\end{equation}

\subsection{Weyl group}

Let $\bW=\bW(G)$ be the Weyl group of $G=\RU_{2r+1}$. Denote by $1$ the unit element in $\bW$. For $w\in \bW$, denote $C(w)=BwB$. The Bruhat order on $\bW$ is defined as follows. Given $w_1,w_2\in \bW$, then $w_1\le w_2$ if and only if $C(w_1)\subset \ov{C(w_2)}$. For $w\in \bW$, we denote $\Omega_w=\coprod_{w'\ge w}C(w')$. Then $C(w)$ is closed in $\Omega_w$ and $\Omega_w$ is open in $G$.

Let $B(G)$ be the subset of $\bW$ which support Bessel functions, i.e., $w\in B(G)$ if and only if for every simple root $\alpha\in \Delta$, if $w\alpha>0$, then $w\alpha$ is also simple. Let $w_\ell=J_{2r+1}\in G$, which represents the longest Weyl element of $G$.  It is well-known that $w\in B(G)$ if and only if $w_\ell w$ is the longest Weyl element of the Levi subgroup of a standard parabolic subgroup of $G$. For $w\in B(G)$, let $P_w=M_wN_w$ be the corresponding parabolic subgroup such that $w_\ell w=w_\ell^{M_w}$, where $M_w$ is the Levi subgroup of $P_w$ and $w_\ell^{M_w}$ is the longest Weyl element of $M_w$. Let $\theta_w$ be the subset of $\Delta$ which consists all simple roots in $M_w$. Then we have the relation 
$$\theta_w=\wpair{\alpha\in \Delta| w\alpha>0}\subset \Delta.$$
The assignment $w\mapsto \theta_w$ is a bijection between $B(G)$ and subsets of $\Delta$. Moreover, it is known that the assignment $w\mapsto \theta_w$ is order-reversing, i.e., $w'\le w$ if and only if $\theta_w\subset \theta_{w'}$, see \cite[Proposition 2.11]{CPSS}. For example, we have $\theta_{w_\ell}=\emptyset$ and $\theta_{1}=\Delta$.

 Given a subset $\theta\subset \Delta$, we will write the corresponding Weyl element in $B(G)$ by $w_\theta$. 

\begin{lem}\label{lem2.4}
For every $n$ with $1\le n\le r$, we have $\tilde w_n=w_{\Delta-\wpair{\alpha_n}}$.
\end{lem}
\begin{proof}
We have $$w_\ell \tilde w_n=\begin{pmatrix}J_n &&\\ &J_{2(r-n)+1}&\\&& J_n \end{pmatrix},$$
which is the longest Weyl element in the Levi subgroup 
$$M_{\tilde w_n}=\begin{pmatrix}a&&\\ &b&\\ &&a^* \end{pmatrix},a\in \GL_n, b\in \RU_{2(r-n)+1}.$$
Thus $\theta_{\tilde w_n}=\Delta-\wpair{\alpha_n}.$
\end{proof}

Given $w,w'\in B(G)$ with $w>w'$, define (following Jacquet \cite{J})
$$d_B(w,w')=\max\wpair{m| \textrm{ there exist }w'_i\in B(G) \textrm{ with } w=w'_m>w'_{m-1}>\dots>w'_0=w'}.$$
The number $d_B(w,w')$ is called the Bessel distance of $w,w'$. By \cite[Proposition 2.1]{CPSS} and Lemma \ref{lem2.4}, the set of elements in $B(G)$ which has Bessel distance 1 with the element $1\in B(G)$ are $\wpair{\tilde w_n, 1\le n\le r}$, i.e.,
\begin{equation}\label{eq2.3}\wpair{w| d_B(w,1)=1}=\wpair{\tilde w_n| 1\le n\le r}.\end{equation}

For $w,w'\in \bW$ with $w<w'$, we denote by $[w,w']$ the closed Bruhat interval $\wpair{w''\in \bW| w\le w''\le w'}$.
\begin{lem}\label{lem2.5}
Given an integer $n$ with $1\le n\le r$. The set $w\in \bW$ with $C(w)\subset P_n\tilde w_n P_n$ is a closed Bruhat interval $[w_{\min},w_{\max}]$ with $w_{\min}=\tilde w_n$ and $w_{\max}=w_{\ell}^{L_n}\tilde w_n$, where $\tilde w_{\ell}^{L_n}$ is the long Weyl element of $L_n$.
\end{lem}
The proof is the same as the proof of \cite[Lemma 2.5]{Zh} and we omit the details here. 

\subsection{Cogdell-Shahidi-Tsai's theory on partial Bessel functions} In this subsection, we review certain basic properties of partial Bessel functions developed by Cogdell-Shahidi-Tsai recently in \cite{CST}.

For $w\in B(G)$, we denote 
$$A_w=\wpair{a\in A| \alpha(a)=1 \textrm{ for all }\alpha\in \theta_w}.$$
The set $A_w$ is in fact the center of $M_w$.
\begin{thm}[Cogdell-Shahidi-Tsai] \label{thm2.6} Let $\omega$ be a character of $E^1$.
\begin{enumerate}
\item Let $w\in \bW$, $m>0$ and $f\in C_c^\infty(\Omega_w,\omega)$. Suppose $B_m(aw,f)=0$ for all $a\in A_w$. Then there exists $f_0\in C_c^\infty(\Omega_w-C(w), \omega)$ which only depends on $f$, such that for sufficiently large $m$ depending only on $f$, we have $B_m(g,f)=B_m(g,f_0)$ for all $g\in G$. 
\item Let $w\in B(G)$. Let $\Omega_{w,0}$ and $\Omega_{w,1}$ be $N\times N$ and $A$-invariant open sets of $\Omega_w$ such that $\Omega_{w,0}\subset \Omega_{w,1}$ and $\Omega_{w,1}-\Omega_{w,0}$ is a union of Bruhat cells $C(w')$ such that $w'$ does not support a Bessel function, i.e., $w'\notin B(G)$. Then for any $f_1\in C_c^\infty(\Omega_{w,1}, \omega)$ there exists $f_0\in C_c^\infty(\Omega_{w,0},\omega)$ such that for all sufficiently large $m$ depending only on $f_1$, we have 
$B_m(g,f_0)=B_m(g,f_1),$ for all $g\in G$.
\end{enumerate}
\end{thm}
\begin{proof}
Part (1) is \cite[Lemma 5.13]{CST} and part (2) is \cite[Lemma 5.14]{CST}. Note that although the results in \cite{CST} were proved in a different setting, their proof is in fact general enough to include our case.
\end{proof}

\begin{cor}\label{cor2.7}
Let $f,f_0\in C_c^\infty(G,\omega)$ with $W^f(I_{2r+1})=W^{f_0}(I_{2r+1})=1$. Then there exist functions $f_{\tilde w_i}\in C_c^\infty(\Omega_{\tilde w_i},\omega)$ for all $i$ with $1\le i\le r$ such that for sufficiently large $m$ (depending only on $f,f_0$) we have 
$$B_m(g,f)-B_m(g,f_0)=\sum_{i=1}^r B_m(g,f_{\tilde w_i}),\forall g\in G.$$
\end{cor}
This is the analogue of \cite[Proposition 5.3]{CST}. The proof is quite similar to the proof of \cite[Proposition 5.3]{CST} and is almost identical with the proof of \cite[Corollary 2.7]{Zh}. Because of its importance, we repeat the argument in the following.
\begin{proof}
For $m$ large (depending only on $f,f_0$), we have $B_m(t,f)=B_m(t,f_0)$ by Corollary \ref{cor2.3}. By Eq.(\ref{eq2.1}), we have $B_m(b,f)-B_m(b,f_0)=0$ for all $b$ in the Borel subgroup $B$. Thus by Theorem \ref{thm2.6}(1), there exists a function $f_1\in C_c^\infty(G-B,\omega)$ such that $B_m(g,f)-B_m(g,f_0)=B_m(g,f_1)$ for all $g\in G$ and $m$ large. Denote 
$$\Omega_1'=\bigcup_{w\in B(G),w\ne 1}\Omega_w=\bigcup_{w\in B(G), d_B(w,1)=1}\Omega_w=\bigcup_{1\le i\le r}\Omega_{\tilde w_i}.$$
See Eq.(\ref{eq2.3}) for the last equality. By Theorem \ref{thm2.6}(2), there exists a function $f_1'\in C_c^\infty(\Omega_1',\omega)$ such that 
$$B_m(g,f)-B_m(g,f_0)=B_m(g,f_1)=B_m(g,f_1').$$
Using a partition of unity argument as in \cite{J} and \cite[Proposition 5.3]{CST}, there exist functions $f_{\tilde w_i}\in C_c^\infty(\Omega_{\tilde w_i}, \omega)$ such that 
$$f_1'=\sum_i f_{\tilde w_i}.$$
Then we have $$B_m(g,f)-B_m(g,f_0)=\sum_i B_m(g,f_{\tilde w_i}),$$
for all $g\in G$ and sufficiently large $m$ (depending only on $f$ and $f_0$).
\end{proof}

\subsection{The local converse theorem} We now state the main theorem of this paper.
\begin{thm}\label{thm2.8}
Let $\pi,\pi_0$ be two irreducible generic supercuspidal representations of $\RU_{2r+1}$ with the same central character $\omega$. If $\gamma(s,\pi\times \tau,\psi)=\gamma(s,\pi_0\times \tau,\psi)$ for all irreducible generic representations $\tau$ of $G_n=\GL_n(E)$ and for all $n$ with $1\le n\le r$, then $\pi\cong \pi_0$.
\end{thm}
\noindent \textbf{Remarks:} 1, The gamma factors we used are the gamma factors defined by the local functional equation in Proposition \ref{prop1.1}. It is expected that these gamma factors are the same as the local gamma factors defined by Langlands-Shahidi method. In particular, the local gamma factors should satisfies multiplicativity properties. If this is true, one only need to use the local gamma factors of $\RU_{2r+1}$ twist by irreducible supercuspidal representations of $\GL_n(E)$ for $1\le n\le r$.\\
2, Note that there is only a single orbit of generic characters of $N$ under the torus $A$ action. Thus if $\pi$ is generic with respect to a given nontrivial additive character $\psi$ of $E$, then it is generic with respect to any nontrivial additive character of $E$. Thus we can assume the additive character $\psi$ is unramified.\\

In the rest of this subsection, we outline our proof of the above theorem. We fix two irreducible generic supercuspidal representations $\pi,\pi_0$ of $\RU_{2r+1}$. We denote by $\CC(0)$ the condition that $\pi$ and $\pi_0$ have the same central character (say $\omega$). For a positive integer $n$ with $1\le n\le r$, we define the condition $\CC(n)$ for $\pi,\pi_0$ inductively:
  \begin{align*} & \textrm{ the condition } \CC(n-1) \textrm{ and}\\
 &\gamma(s,\pi\times \tau, \psi)=\gamma(s,\pi_0\times \tau,\psi),\\
  & \textrm{ for all irreducible generic representations } \tau \textrm{ of } \GL_n(E).
  \end{align*}

Since the representation $\pi$ is supercuspidal, we have $\CM(\pi)\subset C_c^\infty(G,\omega)$, where $\omega$ is the central character of $\pi$ as usual. We can consider the linear functional $\CM(\pi)\ra C^\infty(G,\psi_U,\omega), f\mapsto W^f$ defined by 
$$W^f(g)=\int_N\psi_N^{-1}(u)f(ug)du.$$
Since $\pi$ is $\psi_N$-generic, the above linear functional is nonzero. Thus we can take $f\in \CM(\pi)$ such that $W^f(I_{2r+1})=1$. Note that for $f\in \CM(\pi)$, we actually have $W^f\in \CW(\pi,\psi_N)$. Recall that we have defined the partial Bessel function $B_m(g,f):$
$$B_m(g,f)=\frac{1}{\vol(N_m)}\int_{N_m}W^f(gu')\psi_m^{-1}(u')du'.$$
The partial Bessel function $B_m(~,f)$ is also in the Whittaker model $\CW(\pi,\psi_N)$.\\

Fix $f\in \CM(\pi), f_0\in \CM(\pi)$ such that $W^f(I_{2r+1})=W^{f_0}(I_{2r+1})=1$. The condition $\CC(0)$ implies that $f,f_0\in C_c^\infty(G,\omega)$. By Corollary \ref{cor2.7}, we have the expansion
$$B_m(g,f)-B_m(g,f_0)=\sum_{i=1}^r B_m(g,f_{\tilde w_i}), $$
for sufficiently large $m$. 

In $\S$4, by induction, we will show that the condition $\CC(n)$ implies that 
$$B_m(g,f)-B_m(g,f_0)=\sum_{i=n+1}^r B_m(g,f'_{\tilde w_i}),$$
for some $f_{\tilde w_i}\in C_c^\infty(\Omega_{\tilde w_i},\omega)$. Thus the condition $\CC(r)$ implies that $B_m(g,f)=B_m(g,f_0)$ for all $g\in G$ and sufficiently large $m$. This implies that $\pi\cong \pi_0$ by the uniqueness of Whittaker models and the irreducibility of $\pi,\pi_0$.

\section{Preparations of the proof}
We fix an unramified additive character $\psi$ of $E$.
\subsection{Sections of induced representations}
Let $n$ be a positive integer with $1\le n\le r$. Denote by $\ov{V}_{2n}$ the opposite of $V_{2n}$, i.e., $$\ov{V}_{2n}=\wpair{\bar u(x):=\begin{pmatrix}I_n&\\ x&I_n \end{pmatrix}\in \RU_{2n}}.$$

Let $i$ be positive integer and $H_i$ be the subgroup of $\RU_{2r+1}$ defined in $\S$2.1. Denote
$$\ov{V}_{2n,i}=\wpair{\bar u (x)\in \ov{V}_{2n}| \begin{pmatrix}I_n&&\\ &I_{2r-2n+1}&\\ x&&I_n \end{pmatrix}\in H_i}.$$

Then $\ov{V}_{2n,i}$ is an open compact subgroup of $\ov{V}_{2n}$.

Let $(\tau,V_\tau)$ be an irreducible $\psi_{Z_n}^{-1}$-generic representation of $G_n$. For $v\in V_\tau,$ we consider the $V_\tau$-valued function $f_s^{i,v}$ on $\RU_{2n} $ defined by 
$$f_s^{i,v}(g)=\left\{\begin{array}{lll}|\det(a)|_E^{s+\frac{n-1}{2}} \tau(a)v, & \textrm{if }g=u(x)m(a)\bar u, u(x)\in V_{2n}, a\in G_n, \bar u\in \ov{V}_{2n,i},\\
0, & \textrm{otherwise.} \end{array}\right.$$
\begin{lem}
For any $v\in V_\tau$, there exists an integer $i_0(v)$ such that if $i\ge i_0(v)$, the function $f_s^{i,v}$ defines an element in $I(s,\tau)$.
\end{lem}
The proof is similar, and in fact easier than the proof in the $\Sp_{2n}$ case, which is given in \cite[Lemma 3.2]{Zh}.

 As in $\S$1.1, we fix a nonzero $\psi_{Z_n}^{-1}$-Whittaker functional $\lambda$ of $\tau$ and consider the $\BC$-valued function $\xi_s^{i,v}(g,a)=\lambda(\tau(a)f_s^{i,v}(g))$ on $\RU_{2n}\times G_n$. Then $\xi_s^{i,v}\in V(s,\tau,\psi^{-1})$. From the definition of $f_s^{i,v}$, we get 
\begin{equation}\label{eq3.1}\xi_s^{i,v}(g,I_n)=\left\{\begin{array}{lll}|\det(a)|_E^{s+\frac{n-1}{2}} W_v(a), & \textrm{if }g=u(x)m(a)\bar u, u(x)\in V_{2n}, a\in G_n, \bar u\in \ov{V}_{2n,i},\\
0, & \textrm{otherwise.} \end{array}\right.\end{equation}
Here $W_v(a)=\lambda(\tau(a)v)$ is the Whittaker function of $\tau$ associated with $v$.

We then consider $\tilde \xi_{1-s}^{i,v}=M(s,\tau)\xi_s^{i,v}\in V(1-s,\tau^*,\psi^{-1})$.

\begin{lem}
Let $D$ be an open compact subset of $V_{2n}$. Then there exists an integer $I(D,v)\ge i_0(v)$ such that for all $i\ge I(D,v)$, we have $\tilde\xi_{1-s}^{i,v}(w_n x)=\vol(\ov{V}_{2n,i})v $ for all $x\in D$.
\end{lem}
The proof is the same as in the $\Sp_{2n}$-case, see \cite[Lemma 3.3]{Zh}. 

Since $\tilde \xi_{1-s}^{i,v}\in V(1-s,\tau^*,\psi^{-1})$, we then have 

\begin{equation}\label{eq3.2}
\tilde \xi_{1-s}^{i,v}(m(a)w_n x,I_n)=\vol(\ov{V}_{2n,i})|\det(a)|_E^{1-s+\frac{n-1}{2}}W_v^*(a),
\end{equation}
for all $x\in D,$ and $i\ge I(D,v)$. Here $W^*_v$ is the Whittaker function of the representation $\tau^*$ associated with $v$.

\subsection{A result of Jacquet-Shalika}
\begin{prop}\label{prop3.3}
Let $W'$ be a smooth function on $\GL_n(E)$ which satisfies $W'(ug)=\psi_{Z_n}(u)W'(g)$ for all $u\in Z_n,g\in \GL_n(E)$, and for each integer $k$, the set $\wpair{g\in \GL_n(E)| W'(g)\ne 0, |\det(g)|=q_E^k}$ is compact modulo $Z_n$. Assume, for all irreducible generic representation $\tau$ of $\GL_n(E)$ and all Whittaker functions $W\in \CW(\tau,\psi_{Z_n}^{-1})$, the integral 
$$\int_{Z_n\setminus G_n}W'(g)W(g)|\det(g)|^{-s-k}$$
vanishes for $\Re(s)\ll 0$, where $k$ is a fixed number. Then $W'=0$.
\end{prop}
This is a consequence of \cite[Lemma 3.2]{JS}. See \cite[Corollary 2.1]{Chen} for a proof of the current form of the above proposition.

\section{Proof of the local converse theorem}
Let $\psi$ be an unramified additive character of $E$ and let $\pi,\pi_0$ be two irreducible supercuspidal $\psi_N$-generic representations of $\RU_{2r+1}$ with the same central character $\omega$. Fix $f\in \CM(\pi),f_0\in \CM(\pi_0)$ such that $W^f(I_{2r+1})=W^{f_0}(I_{2r+1})=1$.

The main result of this section is the following
 
\begin{thm}\label{thm4.1}
Given an integer $n$ with $0\le n\le r$. Then condition $\CC(n)$ implies that there exist functions $f_{\tilde w_i}\in C_c^\infty(\Omega_{\tilde w_i},\omega)$, $n+1\le i\le r$ such that 
$$B_m(g,f)-B_m(g,f_0)=\sum_{i=n+1}^r B_m(g,f_{\tilde w_i}),$$
for all $g\in G$ and all sufficiently large $m$ depending only on $f,f_0$.
\end{thm}

If we take $n=r$ in Theorem \ref{thm4.1}, we see that $\CC(r)$ implies $$B_m(g,f)=B_m(g,f_0)$$
for all $g\in G$ and $m$ large. Thus we get $\pi\cong \pi_0$ by the uniqueness of Whittaker models and the irreducibility of $\pi,\pi_0$.

 The idea of the proof of Theorem \ref{thm4.1} is the same as in the $\Sp_{2r}$ case considered in \cite[Proposition 4.1]{Zh}. The only difference is that the local zeta integrals for $\RU_{2r+1}$ and $\Sp_{2r}$ are different.

We will prove Theorem \ref{thm4.1} by induction. Note that the base case when $n=0$ is proved in Corollary \ref{cor2.7}. We assume the following 

\textbf{Inductive Hypothesis:}  Let $n$ be a positive integer with $1\le n\le r$. We assume that the condition $\CC(n-1)$ implies that there exist functions $f_{\tilde w_i}\in C_c^\infty(\Omega_{\tilde w_i},\omega)$ ($n\le i\le r$) such that 
$$B_m(g,f)-B_m(g,f_0)=\sum_{i=n}^r B_m(g,f_{\tilde w_i}).$$

Recall that $P_n$ is the standard parabolic subgroup of $G$ with Levi $L_n$. The simple roots in $L_n$ is the set $\wpair{\alpha_1,\dots,\alpha_{n-1}}$. 
Denote $$S_n^-=\wpair{\begin{pmatrix}I_n &x&y\\ &I_{2r-2n+1}& x'\\ &&I_n \end{pmatrix}\in G}$$
and 
$$S_n^+=\wpair{\begin{pmatrix}u_1&&\\ &u_2&\\ &&u_1' \end{pmatrix}\in G, u_1\in Z_n, u_2\in N_{2r-2n+1}}.$$
Note that $N=S_n^+\cdot S_n^-$. By \cite[p. 12]{Ca}, the product map 
$$P_n\times \wpair{\tilde w_n}\times S_n^-\ra P_n \tilde w_n P_n$$
induces an isomorphism.

Let $m$ be a positive integer, we have defined $N_m$ in $\S$2.1.

For $a\in \GL_n(E)$, denote by $\bt_n(a)$ the element 
$$\begin{pmatrix}a &&&\\ &1&&\\ &&\dots& \\&&&1 \end{pmatrix}^\wedge \in G.$$

\begin{lem}\label{lem4.2}
\begin{enumerate}
\item We have 
$$B_m(a^{\wedge},f)-B_m(a^\wedge,f_0)=0, \forall a\in \GL_r(E),$$
for large enough $m$ depending only on $f,f_0$.
\item We have $P_n\tilde w_n P_n \cap \Omega_{\tilde w_i}=\emptyset$ for $i\ge n+1$. In particular, we have $$B_m(g,f_{\tilde w_i})=0,$$
for all $g\in P_n\tilde w_n P_n$ and $i\ge n+1$, and thus 
$$B_m(g,f)-B_m(g,f_0)=B_m(g,f_{\tilde w_n}),g\in P_n\tilde w_n P_n,$$
for sufficiently large $m$ depending only on $f,f_0$.
\item Let $f_{\tilde w_n}\in C_c^\infty(\Omega_{\tilde w_n},\omega)$ be as in the inductive hypothesis. For sufficiently large $m$ depending only on $f_{\tilde w_n}$ (and hence only on $f,f_0$), we have 
$$B_m(\bt_n(a)\tilde w_n u_0, f_{\tilde w_n})=0,$$
for all $a\in \GL_n(E)$ and $u_0\in S_n^--(N_m\cap S_n^-)$.
\item For a fixed $m$ and each integer $k$, the set $\wpair{a\in \GL_n(E)| B_m(\bt_n(a)\tilde w_n, f_{\tilde w_n})\ne 0, |\det(a)|=q_E^k}$ is compact modulo $Z_n$.
\end{enumerate}
\end{lem}

\begin{proof} The proof of this lemma is the same as the proof of \cite[Lemma 4.3]{Zh}. For later use, we will repeat the proof of (2) here. 

Suppose that $P_n\tilde w_n P_n\cap \Omega_{\tilde w_i}\ne \emptyset$ for some $i\ge n+1$. Then there exists a $w\in \bW$ such that $C(w)\subset P_n\tilde w_n P_n\cap \Omega_{\tilde w_i}$. By Lemma \ref{lem2.5}, we have 
$$w_{\max}=w_{\ell}^{L_n}\tilde w_n\ge w\ge \tilde w_i.$$
A matrix calculation shows that 
$$w_\ell w_{\max}=\begin{pmatrix}I_n&&\\ &J_{2r-2n+1}&\\ &&I_n \end{pmatrix},$$
which is the long Weyl element of the Levi subgroup $M_{\max}\cong G_1^n\times \RU_{2r-2n+1}$. Thus $w_{\max}\in B(G)$. From the matrix form of $w_\ell w_{\max }$, we can also read that
$$\theta_{w_{\max}}=\wpair{\alpha_i| n+1\le i\le r}.$$
On the other hand, we have $\theta_{\tilde w_i}=\Delta-\wpair{\alpha_i}$. Thus for $i\ge n+1$, we cannot have $\theta_{w_{\max}}\subset \theta_{\tilde w_i}$. Since the map $w\mapsto \theta_w$ is order-reversing, we cannot have $w_{\max}\ge \tilde w_i$. Contradiction. This proves (2).
\end{proof}

\begin{prop}\label{prop4.3}
Under the inductive hypothesis, the condition $\CC(n)$ implies that 
$$B_m(\bt_n(a)\tilde w_n, f_{\tilde w_n})=0,$$
for all $a\in \GL_n(E)$ and sufficiently large $m$ depending only on $f,f_0$.
\end{prop}
In the following proof, we will write $j(g)=w_{n,r-n}gw_{n,r-n}^{-1}$ for $g\in G$. Then we have $\tilde w_n=j(w_n)$ and $$\bt_n(a)=j(\diag(1,\dots,1,a)^\wedge), a\in \GL_n(E).$$
Recall that we have identified $\RU_{2n}$ as a subgroup of $G=\RU_{2r+1}$ and thus $w_n\in \RU_{2n}$ can be viewed as an element in $G$. For an $a\in \GL_n(E)$, the element $m(a)=\diag(a,a^*)$ is identified with 
$$ \diag(1,\dots,1,a)^\wedge.$$
\begin{proof}
Let $m$ be sufficiently large integer such that the inductive hypothesis and Lemma \ref{lem4.2} hold.

Let $(\tau,V_\tau)$ be an irreducible generic representation of $G_n=\GL_n(E)$ and let $v\in V_\tau$.

Consider the open compact subset $V_{2n,m}$ of $V_{2n}$ defined by $V_{2n,m}=\wpair{u(y): j(u(y))\in H_m}$. Recall that we have identified $\RU_{2n}$ as a subgroup of $\RU_{2r+1}$ and thus it makes sense to talk $j(u(y))$.

 Let $i$ be an integer with $i\ge \max\wpair{m, i_0(v),I(V_{2n,m},v)}$. Then we can consider the section $\xi_s^{i,v}\in V(s,\tau,\psi^{-1})$ and $\tilde \xi_{1-s}^{i,v}=M(s,\tau)\xi_s^{i,v}$ defined in $\S$3.1.

We compute the integral $\Psi(W^f_m, \xi_s^{i,v})$. Since $ V_{2n}M_{2n}\ov{V}_{2n}$ is dense in $\RU_{2n}$, we will replace $N_{2n}\setminus \RU_{2n}$ by $N_{2n}\setminus V_{2n}M_{2n}\ov{V}_{2n}\cong Z_n\setminus G_n \times \ov{V}_{2n}$ in the integral $\Psi(W_f^m,\xi_s^{i,v}).$ Note that for $g=u(x) m(a)\bar u\in V_{2n}M_{2n}\ov{V}_{2n}$, we can take the quotient Haar measure on $ Z_n\setminus G_n \times \ov{V}_{2n}$ by $dg=|\det(a)|_E^{-n}d\bar u da$.

Thus 
\begin{align*}
&\Psi(W_m^f,\xi_s^{i,v})\\
=&\int_{Z_n\setminus G_n}\int_{\ov{V}_{2n}} \int_{\Mat_{(r-n)\times n}}W_m^f\left(j\left(\begin{pmatrix}I_{r-n} &x\\ &I_{n} \end{pmatrix}^\wedge m(a)\bar u \right)\right)\xi_s^{i,v}(m(a)\bar u,I_n)|\det(a)|_E^{-n}dx d\bar u da.
\end{align*}
By Eq.(\ref{eq3.1}), we have 
\begin{align*}&\Psi(W_m^f,\xi_s^{i,v})\\
=&\int_{Z_n\setminus G_n}\int_{\ov{V}_{2n,i}} \int_{\Mat_{(r-n)\times n}}W_m^f\left(j\left(\begin{pmatrix}I_{r-n} &x\\ &I_{n} \end{pmatrix}^\wedge m(a)\bar u \right)\right) |\det(a)|_E^{s+\frac{-n-1}{2}}W_v(a)dx d\bar u da.
\end{align*}
For $\bar u\in \ov{V}_{2n,i}$, we have $j(\bar u)\in H_i$. Thus by Eq.(\ref{eq2.1}), we have 
$$W_m^f(gj(\bar u))=W_m^f(g), \forall g\in \RU_{2r+1}.$$
Thus we get 
\begin{align*}&\Psi(W_m^f,\xi_s^{i,v})\\
=&\vol(\ov{V}_{2n,i})\int_{Z_n\setminus G_n}\int_{\Mat_{(r-n)\times n}}W_m^f\left(j\left(\begin{pmatrix}I_{r-n} &x\\ &I_{n} \end{pmatrix}^\wedge m(a) \right)\right) |\det(a)|_E^{s+\frac{-n-1}{2}}W_v(a)dx  da.
\end{align*}
One can check that 
$$j\left(\begin{pmatrix}I_{r-n} &x\\ &I_{n} \end{pmatrix}^\wedge m(a) \right) $$
is of the form $X^\wedge$ with $X\in \GL_r(E)$. Thus by Lemma \ref{lem4.2} (1), we have 
$$W_m^f\left(j\left(\begin{pmatrix}I_{r-n} &x\\ &I_{n} \end{pmatrix}^\wedge m(a) \right)\right)=W_m^{f_0}\left(j\left(\begin{pmatrix}I_{r-n} &x\\ &I_{n} \end{pmatrix}^\wedge m(a) \right)\right). $$
Thus we get 
\begin{equation}\label{eq4.1}\Psi(W_m^f,\xi_s^{i,v})=\Psi(W_m^{f_0},\xi_s^{i,v}).\end{equation}
By assumption, we have $\gamma(s,\pi\times \tau,\psi)=\gamma(s,\pi_0\times \tau,\psi)$. By Eq.(\ref{eq4.1}) and the local functional equation, Proposition \ref{prop1.1}, we have 
\begin{equation}\label{eq4.2}\Psi(W_m^f, \tilde\xi_{1-s}^{i,v})=\Psi(W_m^{f_0},\tilde\xi_{1-s}^{i,v}).\end{equation}
We next consider the integral $\Psi(W_m^f, \tilde \xi_{1-s}^{i,v})$. Since $V_{2n}M_{2n}w_nV_{2n}$ is dense in $\RU_{2n}$, we will replace $N_{2n}\setminus \RU_{2n}$ by $N_{2n}\setminus V_{2n} M_{2n}w_n V_{2n}\cong Z_n\setminus G_n w_n V_{2n}$ in the integral $\Psi(W_f^m,\tilde \xi_{1-s}^{i,v})$. We then have 
\begin{align*}
&\Psi(W_m^f,\tilde \xi_{1-s}^{i,v})\\
=&\int_{Z_n\setminus G_n}\int_{V_{2n}}\int_{\Mat_{(r-n)\times n}}W_m^f\left(j\left(\begin{pmatrix}I_{r-n} &x\\ &I_{n} \end{pmatrix}^\wedge m(a)w_n u(y) \right)\right) \\
\quad & \cdot \tilde \xi_{1-s}^{i,v}(m(a)w_n u(y),I_n)|\det(a)|_E^{-n}dydxda.
\end{align*}
There is a similar expression for $\Psi(W_m^{f_0},\tilde \xi_{1-s}^{i,v})$.

We now consider the quantity inside $W_m^f$. We have 
\begin{align*}
&j\left(\begin{pmatrix}I_{r-n} &x\\ &I_{n} \end{pmatrix}^\wedge m(a)w_n u(y) \right)\\
=&j\left(m(a)\begin{pmatrix}I_{r-n} &xa\\ &I_{n} \end{pmatrix}^\wedge w_n u(y) \right)\\
=&j(m(a)w_n \hat u(xa) u(y))\\
=& \bt_n(a)\tilde w_n j(\hat u(xa)) j(u(y)),
\end{align*}
where $$\hat u(xa)=w_n^{-1} \begin{pmatrix}I_{r-n} &xa\\ &I_{n} \end{pmatrix}^\wedge w_n=\begin{pmatrix}I_{r-n}&&&xa &\\ &I_n &&&(xa)'\\ &&1&&\\ &&&I_n& \\ &&&&I_{r-n} \end{pmatrix},$$
where $(xa)'$ is uniquely determined by $xa$ such that $\hat u(xa)\in \RU_{2r+1}$. 
We have $$j(\hat u(xa))j(u(y))=\begin{pmatrix}I_{n}&&&(xa)' &y\\ &I_{r-n} &&&(xa)\\ &&1&&\\ &&&I_{r-n}& \\ &&&&I_{n} \end{pmatrix}\in S_n^-.$$
From the description of $P_n \tilde w_n P_n$, we see that $\bt_n(a)\tilde w_n j(\hat u(xa))j(u(y))\in P_n\tilde w_n P_n$. By Lemma \ref{lem4.2}(2), we get 
\begin{align*}
&W_m^f( \bt_n(a)\tilde w_n j(\hat u(xa))j(u(y)))-W_m^{f_0}(\bt_n(a)\tilde w_n j(\hat u(xa))j(u(y)))\\
=&B_m(\bt_n(a)\tilde w_n j(\hat u(xa))j(u(y)),f)-B_m(\bt_n(a)\tilde w_n j(\hat u(xa))j(u(y)),f_0)\\
=&B_m(\bt_n(a)\tilde w_n j(\hat u(xa))j(u(y)),f_{\tilde w_n}).
\end{align*}
By the above discussion and Eq.(\ref{eq4.2}), we have 
\begin{align}
0&=\Psi(W_m^f,\xi_s^{i,v})-\Psi(W_m^{f_0}, \xi_s^{i,v}) \label{eq4.3}\\
&=\int_{Z_n\setminus G_n}\int_{V_{2n}}\int_{\Mat_{(r-n)\times n}}B_m(\bt_n(a)\tilde w_n j(\hat u(xa))j(u(y)),f_{\tilde w_n}) \nonumber\\
&\cdot  \tilde \xi_{1-s}^{i,v}(m(a)w_n u(y),I_n)|\det(a)|_E^{-n}dydxda \nonumber\\
&=\int_{Z_n\setminus G_n}\int_{V_{2n}}\int_{\Mat_{(r-n)\times n}}B_m(\bt_n(a)\tilde w_n j(\hat u(x))j(u(y)),f_{\tilde w_n}) \nonumber\\
&\cdot  \tilde \xi_{1-s}^{i,v}(m(a)w_n u(y),I_n)|\det(a)|_E^{-r}dydxda,\nonumber
\end{align}
where in the last step, we changed variable $x\mapsto xa^{-1}$.

Let $D_m=\wpair{j(\hat u(x))j(u(y))| j(\hat u(x))j(u(y))\in H_m}$. Then for $j(\hat u(x))j(u(y)) \notin D_m$, we have 
$$B_m(\bt_n(a)\tilde w_n j(\hat u(x))j(u(y)),f_{\tilde w_n})=0$$
by Lemma \ref{lem4.2}(3). On the other hand, if $j(\hat u(x))j(u(y)) \in D_m$, we have 
$$B_m(\bt_n(a)\tilde w_n j(\hat u(x))j(u(y)),f_{\tilde w_n})=B_m(\bt_n(a)\tilde w_n ,f_{\tilde w_n}) $$
by Eq.(\ref{eq2.1}). On the other hand, for $ j(\hat u(x))j(u(y)) \in D_m$, we have $u(y)\in V_{2n,m}$. By our choice of $i$ and Eq.(\ref{eq3.2}), we have 
$$ \tilde \xi_{1-s}^{i,v}(m(a)w_n u(y),I_n)=\vol(\ov{V}_{2n,i})|\det(a)|_E^{1-s+\frac{n-1}{2}}W_v^*(a). $$
Now Eq.(\ref{eq4.3}) reads
$$0=\int_{Z_n\setminus G_n} B_m(\bt_n(a)\tilde w_n,f_{\tilde w_n}) |\det(a)|^{-s+\frac{n+1}{2}-r}W_v^*(a)da.$$
Note that this equation holds for any irreducible generic representation $(\tau,V_\tau)$ and any vector $v\in V_\tau$. Thus we get 
$$B_m(\bt_n(a)\tilde w_n, f_{\tilde w_n})=0$$
by Lemma \ref{lem4.2}(4) and Proposition \ref{prop3.3}. This concludes the proof.
\end{proof}

Now we can finish the proof of Theorem \ref{thm4.1} and hence the proof of our main theorem, Theorem \ref{thm2.8}.

\begin{proof}[Proof of Theorem $\ref{thm4.1}$]
Let $w_{\max}=w_\ell^{L_n}\tilde w_n$ be as in the proof of Lemma \ref{lem4.2}. We can check that $$A_{w_{\max}}=E^1\cdot \wpair{\bt_n(\diag(a_1,\dots,a_n)),a_i\in E^\times}.$$
Thus any element $a\in A_{w_{\max}}$ has the form $z\bt_n(a_0)$ with $z\in E^1$ and $a_0=\diag(a_1,\dots,a_n)\in \GL_n(E)$. By Proposition \ref{prop4.3}, we have 
$$B_m(a\tilde w_n,f_{\tilde w_n})=\omega(z)B_m(\bt_n(a_0)\tilde w_n,f_{\tilde w_n})=0.$$

For any $w\in B(G)$ with $\tilde w_n \le w\le w_{\max}$, we can write $w=w'\tilde w_n$ for a Weyl element $w'$ of the Levi $L_n$, which has a representative of the form $\bt_n(b')$ with $b'\in \GL_n$. Since $w\le w_{\max}$, we can check that $A_w\subset A_{w_{\max}}$. Thus an element $a\in A_w$ is also of the form $\bt_n(a_0)$ for certain torus element in $\GL_n(E)$. Following a similar argument as above, Proposition \ref{prop4.3} implies that 
\begin{equation}\label{eq4.4}B_m(aw, f_{\tilde w_n})=0,\end{equation}
for all $a\in A_w$.

Denote $$\Omega_{\tilde w_n}'=\cup_{w\in B(G), w>w_{\max}\atop d_B(w,w_{\max})=1}\Omega_w.$$
Eq.(\ref{eq4.4}) and Theorem \ref{thm2.6} imply that there exists a function $f'_{\tilde w_n}\in C_c^\infty(\Omega_{\tilde w_n}',\omega)$ such that 
$$B_m(g,f_{\tilde w_n})=B_m(g,f'_{\tilde w_n}),$$
for all $g\in G$ and sufficiently large $m$ (which depends on $f_{\tilde w_n}$ and thus can be chosen so that it only depends on $f,f_0$). 

We now classify the set $\wpair{w\in B(G)| w>w_{\max}, d_B(w,w_{\max})=1}$. From the proof of Lemma \ref{lem4.2}, we see that $$\theta_{w_{\max}}=\wpair{\alpha_i| n+1\le i\le r}.$$
Note that the bijection $w\mapsto \theta_w$ from $B(G)$ to the subsets of $\Delta$ is order reversing. Thus if $w>w_{\max}$ and $d_B(w,w_{\max})=1$, there exists an $i$ with $n+1\le i\le r$ such that $\theta_w=\theta_{\max}-\wpair{\alpha_i}$. Denote 
$$w_i'=w_{\theta_{\max}-\wpair{\alpha_i}}, n+1\le i\le r.$$
We then have $$\wpair{w\in B(G)|w>w_{\max}, d_B(w,w_{\max})=1}=\wpair{w_i'|n+1\le i\le r}.$$

We then have 
$$\Omega_{\tilde w_n}'=\cup_{i=n+1}^r \Omega_{w_i'}.$$
By a partition of unity argument, there exists $f_{w_i'}\in C_c^\infty(\Omega_{w_i'},\omega)$ for $n+1\le i\le r$ such that 
$$f_{\tilde w_n}'=\sum_{i=n+1}^r f_{w_i'}.$$
Thus we have 
\begin{equation}\label{eq4.5}B_{m}(g,f_{\tilde w_n})=\sum_{i=n+1}^r B_m(g,f_{w_i'}).\end{equation}

Fix $i$ with $n+1\le i\le r$. Note that $\theta_{\tilde w_i}=\Delta-\wpair{\alpha_i}$, see Lemma \ref{lem2.4}. Since $\theta_{\max}-\wpair{\alpha_i}\subset \Delta-\wpair{\alpha_i}$, we have $w_i'> \tilde w_i$. Thus $\Omega_{w_i'}\subset \Omega_{\tilde w_i}$. The set $\Omega_{w_i'}$ is in fact open in $\Omega_{\tilde w_i}$, see Proposition \cite[prop2.5]{Zh} for example. Thus $C_c^\infty(\Omega_{ w'_i},\omega)\subset C_c^\infty(\Omega_{\tilde w_i},\omega)$. In particular, $f_{w_i'}$ can be viewed as an element of $C_c^\infty(\Omega_{\tilde w_i},\omega)$. We now define 
$$f'_{\tilde w_i}=f_{w_i'}+f_{\tilde w_i}\in C_c^\infty(\Omega_{\tilde w_i},\omega).$$
Now Eq.(\ref{eq4.5}) and the inductive hypothesis imply that 
$$B_m(g,f)-B_m(g,f_0)=\sum_{i=n+1}^r B_m(g,f'_{\tilde w_i}).$$
This finishes the proof of Theorem \ref{thm4.1} and hence the proof of Theorem \ref{thm2.8}.
\end{proof}

\section{Stability of the local gamma factors for $\RU_{2r+1}\times \Res_{E/F}(\GL_1)$}

Using our technique of the proof of the local converse theorem, we can prove the stability of the local gamma factors for $\RU_{2r+1}\times \Res_{E/F}(\GL_1)$:
\begin{thm}\label{thm5.1}
Let $\pi,\pi_0$ be two irreducible supercuspidal generic representation of $\RU_{2r+1}$ with the same central character $\omega$. Then there exists an integer $l=l(\pi,\pi_0)$ such that for any quasi-character $\chi$ of $E^\times$ with $\cond(\chi)>l$, we have 
$$\gamma(s,\pi\times \chi,\psi)=\gamma(s,\pi_0\times \chi,\psi).$$
\end{thm}
Here the local gamma factors are defined from the local functional equation as in Proposition \ref{prop1.1}.\\
\noindent \textbf{Remark:} Stability results of local gamma factors were proved in many different settings. For example, for a quasi-split classical group $G$, the stability of the local gamma factors for $G\times \GL_1$ arising from Langlands-Shahidi method is proved in \cite{CPSS08};  in \cite{CST}, the authors proved the stability of the exterior square local gamma factors for $\GL$ arising from Langlands-Shahidi method. Since it is not known that the local gamma factors defined in Proposition \ref{prop1.1} coincide with the local gamma factors arising from Langlands-Shahidi method, our stability result cannot be covered by the known cases. \\

For simplicity, in this section we write $j(g)=w_{1,r-1}gw_{1,r-1}^{-1}$ as in the proof of Proposition \ref{prop4.3} (for $n=1$).

Before we go to the proof of Theorem \ref{thm5.1}, we consider the following
\begin{lem}\label{lem5.2}
Let $\psi$ be an unramified additive character of $E$. Let $W\in C^\infty(G,\psi_N,\omega)$ with $W(I_{2r+1})=1$. Let $C$ be a constant such that $W$ is right invariant under $K_C^G$. For $m\ge C$, $a\in E$ and $x={}^t\!(x_1,x_2,\dots,x_{r-1})\in \Mat_{(r-1)\times 1}(E)$, we have  
$$W_m\left(\bt_1(a) j\left(\begin{pmatrix}I_{r-1} &x\\ &1 \end{pmatrix}^\wedge \right) \right)=\left\{\begin{array}{lll}W_m(\bt_1(a)), & \textrm{ if } x_i\in \fp_E^{(2i+1)m} ,1\le i\le r-1,\\
0, &\textrm{otherwise.} \end{array}\right.$$
\end{lem}
\begin{proof}
This follows from a standard argument regarding ``root killing". We omit the details, see \cite[Lemma 2.6]{Zh1} for a proof in a similar situation.
\end{proof}

\begin{proof}[Proof of Theorem $\ref{thm5.1}$]
Take $f\in \CM(\pi),f_0\in \CM(\pi_0)$ such that $W^f(I_{2r+1})=W^{f_0}(I_{2r+1})=1$. Since $\pi,\pi_0$ have the same central character $\omega$, by Corollary \ref{cor2.7}, there exist functions
$f_{\tilde w_i}\in C_c^\infty(\Omega_{w_i},\omega)$ such that 
\begin{align}\label{eq5.1}B_m(g,f)-B_m(g,f_0)=\sum_{i=1}^r B_m(g,f_{\tilde w_i}),\end{align}
for all $g\in G, $ and sufficiently large $m$ depending only on $f,f_0$.  By Lemma \ref{lem4.2}(2), we have 
$$B_m(g,f)-B_m(g,f_0)=B_m(g,f_{\tilde w_1}),$$
for $g\in P_1\tilde w_1 P_1$. By Lemma \ref{lem4.2}(3), we have 
\begin{align}\label{eq5.2}
B_m(\bt_1(a)\tilde w_1 u_0 ,f_{\tilde w_1})=0,
\end{align}
for all $a\in E^\times, u_0\in S_1^--(N_m\cap S_1^-)$, and for sufficiently large $m$ (depending only on $f,f_0$).

Let $\chi$ be a quasi-character of $E^\times$ and $k$ be a positive integer such that Eq.(\ref{eq5.1})-(\ref{eq5.2}) hold for $B_k$. Note that $k$ only depends on $f,f_0$.  Now take an integer $m$ such that $m>\max\wpair{\cond(\chi),k}$. Let $i>m$ be a large enough integer such that we can define $\xi_s^i\in V(s,\chi,\psi^{-1})$ as in $\S$3.1. (Note that the space of $\chi$ is one-dimensional and thus it's unnecessary to add $v$ in the notation $\xi_s^{i,v}$ as in $\S$3.1).

As in the proof of Proposition \ref{prop4.3}, we can compute that 
\begin{align*}&\Psi(W_m^f,\xi_s^i)\\
=&\vol(\ov{V}_{2,i})\int_{E^\times}\int_{\Mat_{(r-1)\times 1}}W_m^f\left( j\left(\begin{pmatrix}I_{r-1}&x\\ &1 \end{pmatrix}^{\wedge}m(a) \right) \right)\chi(a)|a|_E^{s-1} dx da\\
=&\vol(\ov{V}_{2,i})\int_{E^\times}\int_{\Mat_{(r-1)\times 1}}W_m^f\left(\bt_1(a) j\left(\begin{pmatrix}I_{r-1}&xa\\ &1 \end{pmatrix}^{\wedge} \right) \right)\chi(a)|a|_E^{s-1} dx da\\
=&\vol(\ov{V}_{2,i})\int_{E^\times}\int_{\Mat_{(r-1)\times 1}}W_m^f\left(\bt_1(a) j\left(\begin{pmatrix}I_{r-1}&x\\ &1 \end{pmatrix}^{\wedge} \right) \right)\chi(a)|a|_E^{s-r} dx da.
\end{align*}
By Lemma \ref{lem5.2}, we have 
$$\Psi(W_m^f,\xi_s^i)=q_E^{-(r^2-r)m}\int_{E^\times}W_m^f(\bt_1(a))\chi(a)|a|_E^{s-r}da.$$

By Corollary \ref{cor2.3} and the assumption $m\ge \cond(\chi)$, we have 
\begin{equation}\label{eq5.3}\Psi(W_m^f,\xi_s^i)=q_E^{-(r^2-r)}\int_{1+\fp_E^m}\chi(a)da=\vol(\ov{V}_{2,i})q_E^{-(r^2-r)m}\vol(1+\fp^m_E,da).\end{equation}
The same application applies to $\Psi(W_m^{f_0},\xi_s^i)$. In particular, we have 
$$\Psi(W_m^f,\xi_s^i)=\Psi(W_m^{f_0},\xi_s^i)=\vol(\ov{V}_{2,i})q_E^{-(r^2-r)m}\vol(1+\fp_E^m,da). $$
Let $\tilde \xi_{1-s}^i=M(s,\chi)\xi_{s}^i$. From the local functional equation, we get 
\begin{equation}\label{eq5.4}\Psi(W_m^f,\tilde \xi_{1-s}^i)-\Psi(W_m^{f_0},\tilde \xi_{1-s}^i)=\Psi(W_m^f,\xi_s^i)(\gamma(s,\pi\times \chi,\psi)-\gamma(s,\pi_0\times \chi,\psi)).\end{equation}

From the proof of Proposition \ref{prop4.3}, we can obtain that 
\begin{align}
&\Psi(W_m^f,\xi_s^i)-\Psi(W_m^{f_0},\xi_s^i)\label{eq5.5}\\
=&\vol(\ov{V}_{2,i})\vol(D_m)\int_{E^\times}B_m(\bt_1(a)\tilde w_1,f_{\tilde w_1})\chi(\bar a^{-1})|a|^{-s+1-r}da.\nonumber
\end{align}
Combining Eq.(\ref{eq5.3})-Eq.(\ref{eq5.5}), we get 
\begin{align}\label{eq5.6}
&\gamma(s,\pi\times \chi,\psi)-\gamma(s,\pi_0\times \chi,\psi)\\
=&q^{(r^2-r+1)m}\vol(D_m)\int_{E^\times}(W_m^f(\bt_1(a)\tilde w_1)-W_m^{f_0}(\bt_1(a)\tilde w_1))\chi(\bar a)^{-1}|a|^{-s+1-r}da.\nonumber
\end{align}
In the above expression $m$ depends on $\cond(\chi)$. To prove the stability, we will show that in the Eq.(\ref{eq5.6}), the integer $m$ can be replaced by $k$, which only depends on $f,f_0$, and thus only depends on $\pi,\pi_0$. 

Since $m\ge k$, by Lemma \ref{lem2.1}(3), we have 
\begin{align*}
&W_m^f(\bt_1(a)\tilde w_1)-W_m^{f_0}(\bt_1(a)\tilde w_1)\\
=&\frac{1}{\vol(N_m)}\int_{N_m}\psi_N^{-1}(u)(W_k^f(\bt_1(a)\tilde w_1u)-W_k^{f_0}(\bt_1(a)\tilde w_1u))du.
\end{align*}
We can split $N_m$ as $S^+_{1,m}S^-_{1,m}$, where $S^+_{1,m}=S^+_1\cap N_m, S^-_{1,m}=S^-_1\cap N_m$. For $u\in N_m$, we can write $u=u^+u^-$ with $u^+\in S^+_{1,m},u^-\in S^-_{1,m}$. Note that we have $\bt_1(a)\tilde w_1u^+=u^+ \bt_1(a)\tilde w_1$. Thus 
$$W_k^f(\bt_1(a)\tilde w_1u^+u^-)=W_k^f(u^+\bt_1(a)\tilde w_1 u^-)=\psi_N(u^+) W_k^f(\bt_1(a)\tilde w_1 u^-). $$
There is a similar expression for $W_k^{f_0}$. We thus get 
\begin{align*}
&W_m^f(\bt_1(a)\tilde w_1)-W_m^{f_0}(\bt_1(a)\tilde w_1)\\
=&\frac{1}{\vol(S^-_{1,m})}\int_{S^-_{1,m}}\psi_N^{-1}(u)(W_k^f(\bt_1(a)\tilde w_1u^-)-W_k^{f_0}(\bt_1(a)\tilde w_1u^-))du^-.
\end{align*}
By Eq.(\ref{eq5.2}), if $u^-\in S^-_{1,m}-S^-_{1,k}$, we have 
$$W_k^f(\bt_1(a)\tilde w_1u^-)-W_k^{f_0}(\bt_1(a)\tilde w_1u^-)=0. $$
Thus 
\begin{align*}
&W_m^f(\bt_1(a)\tilde w_1)-W_m^{f_0}(\bt_1(a)\tilde w_1)\\
=&\frac{1}{\vol(S^-_{1,m})}\int_{S^-_{1,k}}\psi_N^{-1}(u)(W_k^f(\bt_1(a)\tilde w_1u^-)-W_k^{f_0}(\bt_1(a)\tilde w_1u^-))du^-\\
=&\frac{\vol(S^-_{1,k})}{\vol(S^-_{1,m})}( W_k^f(\bt_1(a)\tilde w_1u^-)-W_k^{f_0}(\bt_1(a)\tilde w_1u^-)),
\end{align*}
where in the last step, we used Eq.(\ref{eq2.1}). Thus Eq.(\ref{eq5.6}) can be written as 

\begin{align}\label{eq5.7}
&\gamma(s,\pi\times \chi,\psi)-\gamma(s,\pi_0\times \chi,\psi)\\
=&c_k\int_{E^\times}(W_k^f(\bt_1(a)\tilde w_1)-W_k^{f_0}(\bt_1(a)\tilde w_1))\chi(\bar a)^{-1}|a|^{-s+1-r}da,\nonumber
\end{align}
where $c_k=q^{(r^2-r+1)m}\vol(D_m)\frac{\vol(S^-_{1,k})}{\vol(S^-_{1,m})} $ is a constant only depends on $k$. Now let $l=l(\pi,\pi_0)$ is a positive integer such that 
$$W_k^f(\bt_1(aa_0)\tilde w_1)-W_k^{f_0}(\bt_1(aa_0)\tilde w_1)=W_k^f(\bt_1(a)\tilde w_1)-W_k^{f_0}(\bt_1(a)\tilde w_1) $$
for all $a_0\in 1+\fp^l$. Then by Eq.(\ref{eq5.7}), we get $$\gamma(s,\pi\times \chi,\psi)-\gamma(s,\pi_0\times \chi,\psi)=0 $$
if $\cond(\chi)>l$.
\end{proof}
\noindent\textbf{Remark:} We expect the method developed in \cite{CST} can give a proof of the stability for the local gamma factors of $\RU_{2r+1}\times \Res_{E/F}(\GL_n)$ for general $n$, i.e., for given irreducible generic representations $\pi_1,\pi_2$ of $\RU_{2r+1}$ with the same central character, and irreducible generic representation $\tau$ of $\GL_n(E)$, there exists an integer $l=l(\pi_1,\pi_2,\tau)$ such that for any quasi-character $\chi$ of $E^\times$ with $\cond(\chi)>l$, then 
$$\gamma(s,\pi_1\times (\chi\otimes \tau),\psi)=\gamma(s,\pi_1\times (\chi\otimes \tau),\psi).$$
In general, let $G$ be a classical group over a $p$-adic field $F$. The local gamma factors for generic representations of $G\times \GL_n(F)$ can be defined, either using Langlands-Shahidi method or using integral representation method. One expects stability for such local gamma factors. But to the author's knowledge, this kind stability (when $n>1$) was solved only for $G=\GL_m$ in \cite{JS}, and is still open for any other classical group $G$.


\begin{thebibliography}{XXXX}
\addtocontents{Bibliography}

\bibitem[AGRS]{AGRS}
A. Aizenbud, D. Gourevitch, S. Rallis, G. Schiffmann, {\em Multiplicity one theorems,} Ann. of Math. \textbf{172}(2010), 1407-1434.

\bibitem[Ba95]{Ba95}
E. M. Baruch, {\em Local factors attached to representations of p-adic groups and strong multiplicity one}, Ph.D. thesis, Yale University, Ann Arbor, MI, 1995, http://search.proquest.com/docview/304230667, MR 2692992

\bibitem[Ba97]{Ba97}
E. M. Baruch, {\em On the Gamma factors attached to representations of $U(2,1)$ over a $p$-adic field}, Israel Journal of Math. \textbf{102} (1997), 317-345.

\bibitem[BAS]{BAS}
A. Ben-Artzi, D. Soudry, {\em $L$-functions for $\RU_m\times R_{E/F}GL_n$ $(n\le \left[ \frac{m}{2}\right])$}, in ``Automorphic Forms and $L$-functions I. Global Aspects",  A volume in honor of S. Gelbart, D. Ginzburg, E. Lapid and D. Soudry, eds. Israel, Math. Conf. Proc., Contemporary Mathematics, \textbf{488} (2009), pp. 13-59. 


\bibitem[Ca]{Ca}
W. Casselman, {\em Introduction to the theory of admissible representations of $p$-adic groups}, online notes, available at 
https://www.math.ubc.ca/~cass/research/pdf/p-adic-book.pdf

\bibitem[Ch]{Ch}
J. Chai, {\em Bessel functions and local converse conjecture of Jacquet}, to appear in J.E.M.S.

\bibitem[Chen]{Chen}
J.J. Chen, {\em Then $n\times (n-2)$ local converse theorem for $\GL(n)$ over a $p$-adic field}, J. Number Theory \textbf{120} (2006), 193-205.

\bibitem[CPSS05]{CPSS}
J.W. Cogdell, I.I. Piatetski-Shapiro, F. Shahidi, {\em Partial Bessel functions for quasi-split groups,} Automorphic Representations, L-functions and Applications: Progress and Prospects, Walter de Gruyter, Berlin, 2005, 95-128.

\bibitem[CPSS08]{CPSS08}
J. W. Cogdell, I. I. Piatetski-Shapiro, F. Shahidi, {\em Stability of gamma factors for quasi-split groups,} J. Inst. Math. Jussieu \textbf{7} (2008), 27-66.

\bibitem[CST]{CST}
J.W. Cogdell, F. Shahidi, T-L. Tsai, {\em Local Langlands correspondence for $\GL_n$ and the exterior and symmetric square $\epsilon$-factors,} Duke Math. J. \textbf{166}(2017), 2053-2132

\bibitem[GGP]{GGP}
W.T. Gan, A. B.H. Gross, D. Prasad, {\em Symplectic local root numbers, central critical L-values and restriction problems in the representation theory of classical groups}, Ast\'erisque \textbf{346} (2012), 1-109.

\bibitem[GPS]{GPS}
S. Gelbart, I.I. Piatetski-Shapiro, {\em Automorphic forms and L-functions for the unitary group}. In Lie group representations, II (College Park, Md., 1982/1983), volume 1041 of Lecture Notes in Math., pages 141-184. Springer, Berlin, 1984. 

\bibitem[J]{J}
H. Jacquet, {\em Germs for Kloosterman integrals, a review.} Contemp. Math., Vol \textbf{664} (2016), 182-195.

\bibitem[JLiu]{JLiu}
H. Jacquet, B. Liu, {\em On the local converse theorem for $p$-adic $\GL_n$}, to appear in American Journal of Math.


\bibitem[JS]{JS}
H. Jacquet, J. Shalika, {\em A lemma on highly ramified $\epsilon$-factors,} Math. Ann. \textbf{271} (1985), 319-332.

\bibitem[Jng]{Jng}
D. Jiang, {\em On local $\gamma$-factors}. In: Arithmetic geometry and number theory, Series on Number Theory and Applications, 1, World Sci. Publ., Hackensack, NJ, 2006, 1-28. 

\bibitem[JngS03]{JngS03}
D. Jiang, D. Soudry, {\em The local converse theorem for $\SO_{2n+1}$ and applications}, Annals of Math. \textbf{157}(2003), 743-806.

\bibitem[JngS12]{JngS12}
D. Jiang, D. Soudry, {\em On the local decent from $\GL(n)$ to classical groups}, appendix to ``Self-dual representations of division algebras and Weil groups: a contrast" by D. Prasad and D. Ramakishnan, American Journal of Math, \textbf{134} (2012), 767-772.

\bibitem[Ka]{Ka}
E. Kaplan, {\em Complementary results on the Rankin-Selberg gamma factors of classical groups}, Journal of Number Theory, \textbf{146}(2015) 390-447.

\bibitem[M]{M}
K. Morimoto, {\em On the irreducibility of global descents for even unitary groups and its applications}, to appear in Transactions of A.M.S.

\bibitem[ST]{ST}
D. Soudry, Y. Tanay, {\em On local decent for unitary groups}, Journal of Number Theory, \textbf{ 146} (2015), 557-626.

\bibitem[Ta]{Ta}
B. Tamir, {\em On L-functions and intertwining operators for unitary groups}, Israel Journal of Math. \textbf{73}(1991) no.2, 161-188.

\bibitem[Zh1]{Zh1}
Q. Zhang, {\em Stability of Rankin-Selberg local gamma factors for $\Sp_{2n}, \wt{\Sp}_{2n}$ and $\RU(n,n)$}, Int. Journal of Number Theory, \textbf{13}(2017), 2393-2432.

\bibitem[Zh2]{Zh}
Q. Zhang, {\em A local converse theorem for $\Sp_{2r}$ and $\RU(r,r)$,} accepted for publication in Math. Ann., 
arXiv:1705.01692 



\end{thebibliography}
\end{document}